\documentclass[12pt,reqno] {amsart}
\usepackage{amsfonts,amssymb,amsmath,amsthm}
\usepackage{mathrsfs}
\usepackage{url}
\usepackage{enumerate}
\usepackage{xcolor}
\usepackage{todonotes}
\usepackage{ulem}

\newtheorem{theorem}{Theorem}[section]
\newtheorem*{theorem*}{Theorem}
\newtheorem{lemma}[theorem]{Lemma}

\newtheorem{prop}[theorem]{Proposition}

\theoremstyle{definition}
\newtheorem{defn}[theorem]{Definition}
\newtheorem{rem}[theorem]{Remark}

\author[M. Bhattacharjee]{Monojit Bhattacharjee}
\author[R. Gupta]{Rajeev Gupta}
\author[V. Venugopal]{Vidhya Venugopal}

\address[M. Bhattacharjee]{Department of Mathematics \\
Birla Institute of Technology and Science K K Birla Goa Campus, India
}

\address[R. Gupta]{School of Mathematics and Computer Science\\
Indian Institute of Technology Goa, India}

\address[V. Venugopal]{Department of Mathematics \\
Birla Institute of Technology and Science K K Birla Goa Campus, India
 }

\begin{document}
\title[Wold-type Decomposition and Analytic Models for $2$-isometries]{Multivariable Wold-Type Decomposition and Analytic Models for a class of left-inverse commuting pairs}

\begin{abstract}
This work establishes a multivariable Wold-type decomposition for left-inverse commuting $n$-tuples of bounded operators, built on the hypothesis that each component admits a Wold-type decomposition.
 For pairs of operators, we obtain a complete analytic model: every left-inverse commuting analytic toral $2$-isometric pair is unitarily equivalent to the pair of multiplication operator by co-ordinate functions $(M_{z_1}, M_{z_2})$ acting on some $\mathcal{E}$-valued  Dirichlet-type space $\mathcal D_{\mathcal E}(\mu_1, \mu_2)$ associated with two finite positive operator-valued Borel measures $\mu_1$ and $\mu_2$ on the unit circle.
An explicit functional model is further derived for the non-analytic case. 
\end{abstract}

\maketitle

\section{Introduction}

    The classical Wold decomposition theorem fundamentally characterizes isometries on any Hilbert space by decomposing them into unitary and shift components, see \cite{NF}. Initially, von Neumann introduced it in foundational article on decomposition of isometry, see \cite{von Neumann}. This was later discovered by Herman Wold in \cite{Wold} for stationary processes. The importance of it is resonated in several branches, namely, in dilation theory, invariant subspace theory in analytic spaces, prediction theory and several others.  Due to the central role of Wold decomposition for single isometry, numerous authors have sought Wold-type decompositions in multivariable settings, that is, for commuting tuples $(V_1,\ldots, V_n)$ of bounded operators, see \cite{Popovici-1998, Popovici, BDF-2006,  Burdak-2007, JS, AJS, BL-2022, MM-2023, NZ}. 
    The example discussed in \cite[Example 6.8]{BEKS} and the remark given in \cite[Remark 2.5]{SSS} are highlighting the indispensability of some assumptions on the commuting $n$-tuples ($n > 1$) to expect the wandering subspace property.  
    While addressing the case of $n=2,$ S{\l}oci{\'n}ski (1980) (see \cite[Theorems 3, 4, and 5]{Slocinski}) produced a Wold-type decomposition for some classes of  pair of commuting isometries e.g. doubly commuting isometries etc. 
    In the same article, he also obtained necessary and sufficient conditions for the existence of a Wold-type decomposition in the summand of the form unitary-unitary, unitary-shift, shift-unitary and shift-shift. In \cite{Popovici}, Popovici, generalized this result to arbitrary pair of commuting isometries by replacing the shift-shift part with a weakly bi-shift component.
    Burdak and coauthors examined \(n\)-tuples of commuting isometries with finite-dimensional wandering spaces and described canonical decompositions and classification results in that finite-multiplicity regime, see  \cite{Burdak2013}.
    In \cite{JS}, Jaydeb provides an analogous decomposition for \(n\)-tuples of doubly commuting isometries
\((V_1,\dots,V_n)\). The decomposition has in general \(2^n\) pieces \(\mathcal{H}_\alpha\) \((\alpha\subset\{1,\dots,n\})\),
which are reducing for all \(V_i\) \((i=1,\dots,n)\), and such that
$V_i\big|_{\mathcal{H}_\alpha}\ \text{is unitary if } i\in \alpha, \text{ and }
V_i\big|_{\mathcal{H}_\alpha}\ \text{is a unilateral shift if } i\notin \alpha.$

In this article, we extend this result to left-inverse commuting (see Definition \ref{defn left-inverse comm tuple}) $n$-tuple $\boldsymbol{T}=(T_1,\ldots, T_n)$ of bounded linear operators defined on a Hilbert space $\mathcal{H}$, such that for each $1\leq i \leq n,$ $T_i$ satisfies the Wold-type decomposition. 
For the case of $n=2,$ in Theorem \ref{Model Theorem - Rank 1}, we find that if $\boldsymbol{T}$ is further cyclic analytic toral $2$-isometry (see \eqref{Toral 2-isometry} for the definition) then $\boldsymbol T$ is unitarily equivalent to the pair $(M_{z_1}, M_{z_2})$ on the function space $\mathcal D(\mu_1,\mu_2)$ for some finite positive Borel measures $\mu_1,\mu_2$ on the unit circle $\mathbb T,$ where $\mathcal D(\mu_1, \mu_2)$ are as defined in \cite[Definition 1.2]{SSS}.
  
  In \cite[Theorem 2.4]{SSS}, the authors characterize the class of pairs $(T_1,T_2)$ of cyclic analytic toral $2$-isometries with the assumption that $\ker T_1^*\cap \ker T_2^*$ is a certain kind of wandering subspace (see \cite[Definition 1.5]{SSS}) for the pair $(T_1,T_2)$. Theorem \ref{Model Theorem - Rank 1} is  an improvement over \cite[Theorem 2.4]{SSS} as  it finds the exact class of pairs $(T_1,T_2)$ which satisfy the hypothesis of \cite[Theorem 2.4]{SSS}.
In Section \ref{Dirichlet space - vector valued}, we define vector valued Dirichlet-type space $\mathcal D_{\mathcal E}(\mu_1,\mu_2).$ 
This turns out to be a reproducing kernel Hilbert space in which the set of polynomials is dense and the pair $(M_{z_1}, M_{z_2})$ is a toral $2$-isometry. 
In Theorem \ref{Model Theorem - Higher Rank}, we obtain that $\mathcal D_\mathcal E(\mu_1, \mu_2)$ models the class of left-inverse commuting pair of analytic toral $2$-isometries. In Theorem \ref{Wold-type decomposition for pair- vector valued}, the condition of analyticity is also dropped and an explicit functional model is given for a pair of left-inverse commuting toral $2$-isometries. 

    In what follows, $\mathcal{H}$ shall always denote a separable complex Hilbert space and $\mathcal{L}(\mathcal{H})$ shall denote the collection of all bounded linear operators defined on $\mathcal{H}$. 
     For any left-invertible operator $T,$ let $T'$ denote the Cauchy dual operator and be defined as $T':=T(T^*T)^{-1}$. We use $\mathcal{H}_{\infty}(T)$ to denote $\bigcap_{m \geq 1}T^m\mathcal{H}$ and $\mathcal{E}_T$ will denote the subspace $ker T^*=\mathcal{H}\ominus T\mathcal{H}.$ 
     Clearly, we have $\mathcal{E}_T=\mathcal{E}_{T'}.$ 
    An operator $T\in\mathcal L(H)$ is called \textit{analytic} if $\mathcal H_\infty(T)$ is trivial. 
    For any $T\in\mathcal L(H)$ and a subset $S$ of $H,$ the closed subspace $\bigvee_{m \in \mathbb Z_+}T^m(S)$ shall be denoted by $\mathcal W_T(S).$ Here $\mathbb Z_+$ denotes the set of all non-negative integers.
   
\section{von Neumann Wold-type decomposition for left invertible operators}

    The primary goal of this section is to establish a von Neumann Wold-type decomposition for a class of left invertible operators $(T_1,\ldots, T_n)$ defined on a Hilbert space $\mathcal{H}.$ Assuming that the operator $T_i,$ $1\leq i \leq n,$ has a Wold-type decomposition separately, we demonstrate that the tuple $(T_1,\ldots, T_n)$ jointly possesses the same property. 
    We shall first deal with the case when $n=2$ and use method of induction to prove the result in general. 
    Based on the seminal paper by S. Shimorin \cite{Shimorin}, we revisit the following definition for an arbitrary operator $T$.
\begin{defn}
    A left invertible operator $T$  on a Hilbert space $\mathcal{H}$ admits \textit{Wold-type decomposition} if $\mathcal{H}$ can be represented as $\mathcal{H}=\mathcal{H}_{\infty}(T) \oplus \mathcal W_T(\mathcal{E}_T)$ such that $\mathcal H_\infty (T)$ is a $T$-reducing subspace and $T\vert_{\mathcal{H}_\infty(T)}$ is unitary.  
\end{defn}
The following results are due to S. Shimorin, see \cite[Proposition 2.7]{Shimorin}.
    
\begin{prop}\label{orthogonal complement of hyper-range}
    For any left invertible operator $T\in\mathcal L(H)$, the following are true:
    \begin{enumerate}
        \item $\mathcal{H}_{\infty}(T')^\perp = \mathcal W_T(\mathcal E_T);$
    \item $\mathcal{H}_{\infty}(T)^\perp = \mathcal W_{T'}(\mathcal E_T).$
     \end{enumerate}
\end{prop}
    Proposition \ref{orthogonal complement of hyper-range} leads us to the following significant conclusions, which will help us to prove one of our main results, that is  Theorem \ref{wold-type decomposition for pair}.
\begin{theorem}\cite[Corollary 2.9]{Shimorin}\label{thm_2}
    The operator $T$ admits Wold-type decomposition if and only if $T'$ admits it. In this case $\mathcal{H}_{\infty}(T)=\mathcal{H}_{\infty}(T')$ and $\mathcal W_T(\mathcal{E})=\mathcal W_{T'}(\mathcal{E}).$
\end{theorem}
\begin{theorem}\cite[Corollary 2.11]{Shimorin}\label{thm_1}
    If $T$ is analytic, and $\mathcal{H}_{\infty}(T')$ is reducing for $T',$ then $T$ possesses the wandering subspace property.
\end{theorem} 
The following notation was introduced in \cite{MRV}. We reproduce the definition here, as it will be used extensively in our discussion of the class of left-inverse commuting tuples throughout the article.
\begin{defn}[Left-inverse commuting tuple] \label{defn left-inverse comm tuple}
        A tuple $(T_1,\ldots, T_n)$ of left invertible commuting operators is said to be a \textit{left-inverse commuting tuple} if     \begin{eqnarray*}
      L_iT_j=T_jL_i \quad \mbox{for } 1\leq i\neq j\leq n,
  \end{eqnarray*}
  where $L_j:= (T_j^*T_j)^{-1}T_j^*.$ 
    \end{defn}
    The following fact about left-inverse commuting operator pair will be used in what follows.
\begin{prop} \label{proposition_1} \cite [Corollary 4.2]{MRV}
     Let $(T_1,\ldots, T_n)$ be a left-inverse commuting tuple of operators. Then for $1\leq i \neq j\leq n$, $ker T_i^*$ is $T_j$-reducing.
\end{prop}
     Two notable results of this section are Theorem \ref{wold-type decomposition for pair} and Theorem \ref{wold-type decomposition}. Even though Theorem \ref{wold-type decomposition for pair} is a particular case for Theorem \ref{wold-type decomposition}, it works as a base case for Theorem \ref{wold-type decomposition} and  demonstrates that the class of left-inverse commuting pairs $(T_1, T_2)$ of bounded linear operators possess a joint Wold-type decomposition, under the assumption that each operator individually satisfies the Wold-type decomposition. The subsequent lemmas play a crucial role in establishing Theorem \ref{wold-type decomposition for pair}.
\begin{lemma}\label{H-1 and W-1 are T-2 reducing}
    Let $(T_1, T_2)$ be a pair of left-inverse commuting operators with $T_1$ having the Wold-type decomposition. Then $\bigcap_{m \geq 0}T_1^m\mathcal{H}$ and $\bigvee_{m \geq 0}T_1^m(\mathcal{E}_1)$ are $T_2$-reducing. 
\end{lemma}
\begin{proof}
    Considering that $T_1$ has the Wold-type decomposition, we get 
    \[\mathcal{H}=\mathcal H_\infty(T_1) \oplus \mathcal W_{T_1}(\mathcal{E}_{T_1}),\]
    where $\mathcal H_\infty(T_1)$ is $T_1$-reducing and $T_1\vert_{\mathcal H_\infty(T_1)}$ is unitary. Due to the commutativity of $T_1, T_2$, it is easy to verify that $\mathcal H_\infty(T_1)$ is $T_2$-invariant. Note here that  $x \in \mathcal W_{T_1}(\mathcal{E}_{T_1})$ implies that $x$ is limit of the elements of the form
\begin{align*}
    x^{(n)} = \sum_{k =0}^n T_1^k x_k, ~x_k \in \mathcal{E}_1.
\end{align*}
Hence $T_2x$ will be the limit of the elements of the form $T_2x^{(n)}=\sum_{k=0}^n T_1^k T_2 x_k \in \mathcal W_{T_1}(\mathcal{E}_{T_1})$, where $T_2x_k \in \mathcal{E}_1$ by Proposition \ref{proposition_1}. Thus $\mathcal W_{T_1}(\mathcal{E}_{T_1})$ is $T_2$-invariant which shows that both $\mathcal H_\infty(T_1)$ and $\mathcal W_{T_1}(\mathcal{E}_{T_1})$ are $T_2$-reducing. 
\end{proof}
In what follows, it will be useful to set the notations  
\begin{eqnarray}\label{H-00 and H-10}
    \mathcal{H}_{00}&:=& \bigcap_{m,n \geq 0}T_1^m T_2^n \mathcal{H}\\
    \mathcal{H}_{10}&:=& \bigcap_{n \geq 0}T_2^n \big(\bigvee_{m \geq 0} T_1^m(\mathcal{E}_1)\big).\nonumber
\end{eqnarray}

\begin{lemma}\label{reducing} Let $(T_1, T_2)$ be a left-inverse commuting pair of operators where both  $T_1$ and $T_2$ possess a Wold-type decomposition. Then
     the subspaces $\mathcal{H}_{00}, \mathcal{H}_{10}$ of $\mathcal{H}$ are $T_i$-reducing for $i=1,2$. 
\end{lemma}
\begin{proof}
    The fact that $\mathcal{H}_{00}$ is both $T_1$-invariant and $T_2$-invariant is immediate because of the commutativity of $T_1$ and $T_2$. 
    Consider $x \in \mathcal{H}_{00}$. Then $x = T_1^m T_2^n x_{m,n}$ for some $x_{m,n} \in \mathcal{H},$  $m,n \in\mathbb Z_{\geq 0}$ and using the left-inverse commuting property, we get
\begin{align*}
   L_2x=L_2  T_1^m T_2^n x_{m,n} = T_1^m L_2T_2^nx_{m,n} = T_1^m T_2^{n-1} x_{m,n} \in \mathcal{H}_{00}. 
\end{align*}
    This shows that $\mathcal{H}_{00}$ is $L_2$-invariant. As $T_2\vert_{\mathcal H_\infty(T_2)}$ is unitary, $\mathcal{H}_{00} \subseteq \mathcal H_\infty(T_2),$ and $T_2L_2$ projects onto the range of $T_2$, it follows that 
\begin{align*}
    T_2^*x = T_2^* T_2L_2x = L_2x \in \mathcal{H}_{00}.
\end{align*} 
Hence $\mathcal{H}_{00}$ is $T_2$-reducing. In a similar way, one can also see that $\mathcal{H}_{00}$ is $T_1$-reducing.

The subspace $\mathcal{H}_{10}$ is clearly  both $T_1$-invariant and $T_2$-invariant. Let $x \in \mathcal{H}_{10}$ be given by $x=T_2^nx_n,~x_n \in \bigvee_{m \geq 0}T_1^m(\mathcal{E}_1)$ for all $n \geq 0$. Then
\begin{align*}
    L_2x=L_2T_2^nx_n = T_2^{n-1}x_n \in \mathcal{H}_{10}, 
\end{align*}
    and this shows that $\mathcal{H}_{10}$ is $L_2$-invariant. Proceeding as in the case of $\mathcal{H}_{00}$, because of the fact that $T_2L_2$ projects onto the range of $T_2$, we get
\begin{align*}
     T_2^*x= T_2^*T_2L_2x =  L_2x \in \mathcal{H}_{10},
\end{align*}
where $T_2\vert_{\mathcal H_\infty(T_2)}$ is unitary. We claim that $\mathcal{H}_{10}$ is $T_1^*$-invariant. It is enough to show that $\mathcal{H}_{10}^{\perp}$ in $\mathcal W_{T_1}(\mathcal E_{T_1})$ is $T_1$-invariant.
Note that $$\mathcal{H}_{10}^{\perp}=\bigvee_{m \geq 0}T_1^m(\mathcal{E}_1) \ominus \bigcap_{n \geq 0}T_2^n(\bigvee_{m \geq 0}T_1^m(\mathcal{E}_1)).$$
 We shall show that $\bigcap_{n \geq 0}T_2^n (\mathcal{H}_{10}^{\perp})=\{0\}$. In fact, $\mathcal{H}_{10}^{\perp} \subseteq \mathcal{H}_1^{(1)}$ implies $\bigcap_{n \geq 0}T_2^n (\mathcal{H}_{10}^{\perp}) \subseteq \bigcap_{n \geq 0}T_2^n (\mathcal{H}_1^{(1)})=\mathcal{H}_{10}$. Also, since $\mathcal{H}_{10}^{\perp}$ is $T_2$-invariant, we get $\bigcap_{n \geq 0}T_2^n(\mathcal{H}_{10}^{\perp}) \subseteq \mathcal{H}_{10}^{\perp}$. This proves the claim that $T_2$ is analytic on $\mathcal{H}_{10}^{\perp}$. Note here that $\bigcap_{n \geq 0}T_2^{'n}(\mathcal{H}_{10}^{\perp})$ is $T_2^{'}$-invariant. For $x \in \bigcap_{n \geq 0}T_2^{'n}(\mathcal{H}_{10}^{\perp})$ we have
\begin{align*}
    T_2^{'*}x = L_2x = L_2T_2T_2^*x = T_2^*x = T_2^*T_2^{'n}x_n = T_2^{'n-1}x_n \in \bigcap_{n \geq 0} T_2^{'n}(\mathcal{H}_{10}^{\perp}),
\end{align*}
where $\bigcap_{n \geq 0}T_2^{'n}(\mathcal{H}_{10}^{\perp}) \subseteq \bigcap_{n \geq 0}T_2^{'n}\mathcal{H} = \bigcap_{n \geq 0}T_2^n \mathcal{H}$ and $T_2 \vert_{\bigcap_{n \geq 0}T_2^n \mathcal{H}}$ is unitary as $T_2$ has the Wold-type decomposition. This concludes that $\bigcap_{n \geq 0}T_2^{'n}(\mathcal{H}_{10}^{\perp})$ is $T_2^{'}$-reducing. Hence, by the Theorem \ref{thm_1}, it is immediate that $T_2$ on $\mathcal{H}_{10}^{\perp}$ has the wandering subspace property. Thus
\begin{align*}
    \mathcal{H}_{10}^{\perp} = \bigvee_{m,n \geq 0} T_1^m T_2^n (\mathcal{E}_1 \cap \mathcal{E}_2).
\end{align*}
    Therefore, it is clear that $\mathcal{H}_{10}^{\perp}$ is $T_1$-invariant and so $\mathcal{H}_{10}$ is $T_1$-reducing. This completes the proof. 
\end{proof}
\begin{lemma} \label{analyticity of T_2}
    Consider $(T_1, T_2)$ to be a left-inverse commuting pair of operators where both  $T_1$ and $T_2$ possess a Wold-type decomposition. Then the operator $T_2$ is analytic on $ \mathcal H_\infty(T_1)\ominus \mathcal{H}_{00}$, that is, 
    $\bigcap_{n \geq 0}T_2^n(\mathcal H_\infty(T_1)\ominus \mathcal{H}_{00})=\{0\}$. 
    
\end{lemma}
\begin{proof}
    Note that  $\mathcal H_\infty(T_1)\ominus \mathcal{H}_{00} \subseteq \mathcal H_\infty(T_1)$, therefore it follows that $$\bigcap_{n \geq 0}T_2^n(\mathcal H_\infty(T_1)\ominus \mathcal{H}_{00}) \subseteq \bigcap_{n \geq 0} T_2^n (\mathcal H_\infty(T_1)) = \mathcal{H}_{00}.$$ 
    From Lemma \ref{reducing}, we have  $\mathcal{H}_{00}$ is $T_2$-reducing. This shows that $$\bigcap_{n \geq 0}T_2^n(\mathcal H_\infty(T_1)\ominus \mathcal{H}_{00}) \subseteq \bigcap_{n \geq 0}T_2^n(\mathcal{H}_{00})^{\perp} \subseteq \mathcal{H}_{00}^{\perp}.$$ 
    This proves the lemma.
\end{proof}
\begin{lemma}
    Given any operator $T \in \mathcal{L}(\mathcal{H})$, we have $\mathcal{H} \ominus T\mathcal{H} = \mathcal H_\infty(T)^{\perp} \cap \big(T(\mathcal H_\infty(T)^{\perp})\big)^{\perp}.$
\end{lemma}
\begin{proof}
    Let $\alpha \in \mathcal H_\infty(T)^{\perp} \cap \big(T(\mathcal H_\infty(T)^{\perp})\big)^{\perp}$. We need to show that $\alpha$ is orthogonal to the range space  $T\mathcal{H}$. 
    Let $x\in H.$ Since  $\mathcal{H}=\mathcal H_\infty(T) \oplus \mathcal H_\infty(T)^{\perp},$  there exist unique $y \in \mathcal H_\infty(T)$ and $z \in \mathcal H_\infty(T)^{\perp}$ such that $x = y+z.$ Since $\mathcal H_\infty(T)$ is $T$-invariant, it follows that  
\begin{align*}
    \langle \alpha, Tx\rangle = \langle \alpha, T(y+z) \rangle = \langle \alpha, Ty\rangle + \langle \alpha, Tz\rangle=0.
\end{align*}
     This implies that $\alpha$ is orthogonal to $T\mathcal{H}$. 

    Conversely, let $\alpha \in \mathcal{H} \ominus T\mathcal{H} $. Since $\alpha \perp T\mathcal{H},$ it trivially follows that $\alpha$ is orthogonal to $\mathcal H_\infty(T)$. Note that, $T(\mathcal H_\infty(T)^{\perp}) \subseteq T\mathcal{H}$ and therefore  $(T\mathcal{H})^{\perp} \subseteq \big(T(\mathcal H_\infty(T)^{\perp})\big)^{\perp}.$ This completes the proof of the lemma. 
\end{proof}

\begin{lemma}\label{orthogonal complement expression}
    Consider $(T_1, T_2)$ to be a left-inverse commuting pair of operators where $T_1$ possesses a Wold-type decomposition. Then $\mathcal H_\infty(T_1)=T_2(\mathcal H_\infty(T_1)) \oplus \bigcap_{m \geq 0}T_1^m(\mathcal{E}_2).$  
\end{lemma}
\begin{proof}
    Clearly, $\bigcap_{m \geq 0}T_1^m(\mathcal{E}_2) \subseteq \bigcap_{m \geq 0}T_1^m\mathcal{H}$. 
    We claim that $\bigcap_{m \geq 0}T_1^m(\mathcal{E}_2) \subseteq \mathcal H_\infty(T_1) \ominus T_2 (\mathcal H_\infty(T_1)).$
    Let $x = T_1^m y_m \in \bigcap_{m \geq 0}T_1^m(\mathcal{E}_2)$, with $y_m \in \mathcal{E}_2,  \forall m \geq 0$. Then for any $z\in T_2(\mathcal H_\infty(T_1)),$ note that $\langle x, T_2 z \rangle = \langle T_2^*x, z \rangle = 0.$
    This proves the claim.  

    To see the reverse inclusion, let $x \in \mathcal H_\infty(T_1) \ominus T_2 (\mathcal H_\infty(T_1))$. For each $m \geq 0,$  there exists $x_m \in \mathcal{H}$ such that $x = T_1^m x_m.$ It is enough to show that $x_m \in \mathcal{E}_2, \forall m \geq 0$. Since $x$ is orthogonal to $T_2(\mathcal H_\infty(T_1)),$ for any $z\in T_2(\mathcal H_\infty(T_1)),$ we get that $\langle x, T_2 z \rangle = 0 = \langle T_2^*x, z\rangle.$
    From Lemma \ref{H-1 and W-1 are T-2 reducing}, we know that $\mathcal H_\infty(T_1)$ is $T_2$-reducing, hence we obtain $T_2^*x = 0.$ Thus $x \in \mathcal{E}_{T_2}$. Since $\mathcal{E}_{T_2}$ is $T_1$-reducing (see Lemma \ref{proposition_1}), observe that 
    \[ L_1^mx=L_1^mT_1^mx_m = x_m \in \mathcal{E}_2, ~\forall m \geq 0.\] 
    This shows that  $\bigcap_{m \geq 0}T_1^m\mathcal{H} \ominus T_2 \bigcap_{m \geq 0}T_1^m\mathcal{H} \subseteq \bigcap_{m \geq 0}T_1^m(\mathcal{E}_2)$ and completes the proof. 
\end{proof} 

\begin{lemma} Let $(T_1,T_2)$ be a pair of left-inverse commuting operators defined on $\mathcal{H}$ such that $T_1$ has the Wold-type decomposition. Then 
\[\mathcal W_1(\mathcal E_1)\ominus T_2(\mathcal W_1(\mathcal E_1))=\mathcal E_2\cap \mathcal W_1(\mathcal E_1)=\mathcal W_1(\mathcal E_1\cap \mathcal E_2).\]  
\end{lemma}
\begin{proof}
 From Lemma \ref{H-1 and W-1 are T-2 reducing}, we get that $W_1(\mathcal E_1)$ is $T_2$-reducing. Hence $T_2(W_1(\mathcal E_1))\subseteq W_1(\mathcal E_1)$ and $T_2^*(W_1(\mathcal E_1))\subseteq W_1(\mathcal E_1).$ The containment $\mathcal E_2\cap W_1(\mathcal E_1)\subseteq W_1(\mathcal E_1)\ominus T_2(W_1(\mathcal E_1))$ follows easily from the definition of $\mathcal E_2.$ To complete the proof of the first equality of the lemma, let $x\in W_1(\mathcal E_1)$ be such that $x$ is orthogonal to $T_2(W_1(\mathcal E_1)).$ Note that for any $y\in W_1(\mathcal E_1),$ $ \langle T_2^*x, y \rangle=\langle x, T_2y \rangle = 0.$ Therefore, we get that $T_2^*x=0$. That is, $x\in\mathcal E_2$ and hence $x\in \mathcal E_2\cap W_1(\mathcal E_1).$  

Since $\mathcal{E}_2$ is $T_1$-reducing, note that $T_1^m(\mathcal{E}_1 \cap \mathcal{E}_2) \subseteq \mathcal{E}_2$ for each $m\geq 0.$ One also  trivially gets that $T_1^m(\mathcal{E}_1 \cap \mathcal{E}_2) \subseteq W_1(\mathcal{E}_1)$. Hence $W_1(\mathcal E_1\cap \mathcal E_2)\subseteq \mathcal E_2\cap W_1(\mathcal E_1).$ To see the reverse containment, let $x\in \mathcal E_2\cap W_1(\mathcal E_1).$ Then, $x\in\mathcal E_2$ and $x$ is the limit of the elements of the form $x^{(n)}= \sum_{k=0}^n T_1^k x_k$, where $x_k \in \mathcal{E}_1$. 
Since $\mathcal{E}_2$ is $T_1$-reducing, we have $L_1^nx^{(n)}=x_n \in \mathcal{E}_2$. Hence $x^{(n)} - x_k$ and therefore $x - x_k$ are in $\mathcal{E}_2$ for each $k\in\mathbb Z_+$. Once again using the fact that $\mathcal{E}_2$ is $T_1$-reducing, we get $L_1^{n-1}(x-T_1^nx_n)=x_{n-1}\in\mathcal E_2$. Continuing this way, we obtain $x_k \in \mathcal{E}_2, ~\forall k.$ This completes the proof of the lemma.
\end{proof}
For a pair of bounded operators $(T_1, T_2),$ set the notations $\mathcal{H}_{01}:=\bigvee_{n \geq 0}T_2^n(\bigcap_{m \geq 0}T_1^m(\mathcal{E}_2))$ and $\mathcal{H}_{11}=\bigvee_{m,n \geq 0}T_1^mT_2^n(\mathcal{E}_1 \cap \mathcal{E}_2)$. In the following theorem, the notations $\mathcal H_{00}$ and $\mathcal H_{10}$ are same as defined in \eqref{H-00 and H-10}.
\begin{theorem}\label{wold-type decomposition for pair}
    Let $(T_1, T_2)$ be a left-inverse commuting  pair of operators defined on $\mathcal{H}.$ Suppose $T_1$ and $T_2$ satisfy the Wold-type decomposition. Then the pair $(T_1,T_2)$ has the joint Wold-type decomposition given by \[\mathcal{H}=\mathcal{H}_{00}\oplus \mathcal{H}_{01} \oplus \mathcal{H}_{10}\oplus \mathcal{H}_{11},\] where for each $i,j=0,1,$ the subspaces $\mathcal H_{ij}$ are $T_1$-reducing and $T_2$-reducing both, and  $T_1$ acts unitarily on $\mathcal{H}_{00},\mathcal{H}_{01}$, while $T_2$ acts unitarily on $\mathcal{H}_{00},\mathcal{H}_{10}.$
\end{theorem}
\begin{proof}
    Since $T_1$ has Wold-type decomposition, $\mathcal H=\mathcal H_\infty(T_1)\oplus \mathcal W_{T_1}(\mathcal E_1),$
    where $\mathcal{H}_\infty(T_1)$ is $T_1$-reducing and $T_1\vert_{\mathcal H_\infty(T_1)}$ is unitary.
    Note here that both the subspaces $\mathcal{H}_\infty(T_1)$ and $\mathcal W_{T_1}(\mathcal E_1)$ are $T_2$-invariant and hence are $T_2$-reducing.
    From Lemma \ref{reducing}, it follows  that $\mathcal{H}_{00}$ reduces both $T_1$ and $T_2$.
     Since  $\mathcal{H}_{00} \subseteq \mathcal{H}_{\infty}(T_i)$ for $i=1,2$ and $T_1$ and $T_2$ both have the Wold-type decomposition, it follows that $T_1$ and $T_2$ both act unitarily on $\mathcal{H}_{00}.$  
     From Lemma \ref{analyticity of T_2}, we get that $\bigcap_{n \geq 0}T_2^n(\mathcal{H}_{\infty}(T_1) \ominus \mathcal{H}_{00})=\{0\}$, thereby proving the analyticity of $T_2$ on $\mathcal{H}_{\infty}(T_1) \ominus \mathcal{H}_{00}.$ 
     Note that $\bigcap_{n \geq 0}T_2^{'n}(\mathcal{H}_{\infty}(T_1) \ominus \mathcal{H}_{00})$ is $T_2'$-invariant. From Theorem \ref{thm_2}, we obtain  $$\bigcap_{n \geq 0}T_2^{'n}(\mathcal{H}_{\infty}(T_1) \ominus \mathcal{H}_{00}) \subseteq \bigcap_{n \geq 0}T_2^{'n}\mathcal{H}=\bigcap_{n \geq 0}T_2^n\mathcal{H}.$$ 
     Since $T_2 \vert_{\mathcal H_\infty(T_2)}$ is unitary, therefore for $x \in \bigcap_{n \geq 0}T_2^{'n}(\mathcal{H}_{\infty}(T_1) \ominus \mathcal{H}_{00})$, we get
\begin{align*}
    T_2'^*x = L_2x= L_2T_2T_2^*x = T_2^*x = T_2^* T_2'^nx_n = T_2'^{n-1} x_n \in \bigcap_{n \geq 0}T_2^{'n}(\mathcal{H}_{\infty}(T_1) \ominus \mathcal{H}_{00}),
\end{align*}
    which shows that $\bigcap_{n \geq 0}T_2^{'n}(\mathcal{H}_\infty(T_1)\ominus \mathcal{H}_{00})$ is $T_2'$-reducing.
    Thus from Theorem \ref{thm_1}, we get 
    \[\mathcal{H}_\infty(T_1)\ominus \mathcal{H}_{00}=\mathcal W_2(\mathcal{H}_\infty(T_1) \ominus \mathcal{H}_{00}).\]
    Therefore, it follows that
    \begin{align}\label{subspace_1}
    \mathcal{H}_\infty(T_1) = \mathcal{H}_{00} \oplus \mathcal W_2(\mathcal{H}_{0}^{(1)}\ominus T_2(\mathcal{H}_{0}^{(1)}))
    = \mathcal H_{00} \oplus \mathcal W_2(\bigcap_{m \geq 0} T_1^m(\mathcal{E}_2))
\end{align}
    where the second equality follows from Lemma \ref{orthogonal complement expression}.

   Next, we consider the subspace  $\mathcal{H}_{01}.$ It can be seen that $\mathcal{H}_{01}$ is both $T_1$-reducing and $T_2$-reducing.  It is noteworthy that
\begin{align*}
    \bigvee_{n \geq 0} T_2^n (\bigcap_{m \geq 0} T_1^m(\mathcal{E}_2)) = \left(\bigvee_{n \geq 0} T_2^n(\mathcal{E}_2)\right) \cap \left(\bigcap_{m \geq 0} T_1^m \mathcal{H}\right).
\end{align*}
   Consequently, $T_1$ is unitary on $\mathcal{H}_{01}$ as well.    
   Similarly, we get
   \begin{align}\label{subspace_2}
   \mathcal W_1(\mathcal{E}_1)  &= \bigcap_{n \geq 0}T_2^n\left(\mathcal W_1(\mathcal{E}_1)\right) \oplus \mathcal W_2\left(\mathcal W_1(\mathcal{E}_1)  \ominus T_2(\mathcal W_1(\mathcal{E}_1)) \right) \nonumber \\
   &= \bigcap_{n \geq 0}T_2^n\left(\mathcal W_1(\mathcal{E}_1)\right) \oplus \mathcal W_{12}(\mathcal{E}_1 \cap \mathcal{E}_2).
\end{align}
  Clearly $\mathcal{H}_{10}$ and $\mathcal{H}_{11}$ are closed subspaces that reduce both $T_1$ and $T_2$. As mentioned above here also we have
\begin{align*}
    \bigcap_{n \geq 0}T_2^n\left(\bigvee_{m \geq 0}T_1^m(\mathcal{E}_1)\right) &=\left( \bigcap_{n \geq 0}T_2^n \mathcal{H}\right) \cap \left(\bigvee_{m \geq 0} T_1^m(\mathcal{E}_1) \right).
\end{align*}
    This implies that $T_2$ on $\mathcal{H}_{10}$ is unitary and  $(T_1,T_2)$ has the joint wandering subspace property on $\mathcal{H}_{11}$.
    Therefore combining \eqref{subspace_1} and \eqref{subspace_2}, we get the following;
\begin{align*}
    \mathcal{H} = \mathcal{H}_{\infty}(T_1) \oplus \mathcal W_{T_1}(\mathcal E_1) 
     = \mathcal{H}_{00} \oplus \mathcal{H}_{01} \oplus \mathcal{H}_{10} \oplus \mathcal{H}_{11},
\end{align*}
    where $T_1$ is unitary on $\mathcal{H}_{00},$ $\mathcal{H}_{01}$ and $T_2$ is unitary on $\mathcal{H}_{00},\mathcal{H}_{10}.$
\end{proof} 
For a commuting $n$-tuple $\boldsymbol{T}:=(T_1,\ldots,T_n)$ of bounded linear operators, we choose to write 
$\boldsymbol{T}^{\boldsymbol m}=T_1^{m_1}\cdots T_n^{m_n}$ for any
$\boldsymbol{m}=(m_1,\ldots,m_n)\in \mathbb Z_{+}^n$, where we adopt the convention $S^0=I$ for any operator $S$. For any
$\boldsymbol{\alpha}=(\alpha_1,\ldots,\alpha_n)\in \{0,1\}^n$, let $\hat{\boldsymbol\alpha}$ be the unique element in $\{0,1\}^n$ such that $\boldsymbol\alpha + \hat{\boldsymbol\alpha}=\boldsymbol 1,$ where $\boldsymbol{1}$ denotes the vector $(1,\ldots, 1)$ in $\{0,1\}^n.$ For any
$\boldsymbol{\alpha}=(\alpha_1,\ldots,\alpha_n)\in \{0,1\}^n$ and $\boldsymbol m=(m_1,\ldots, m_n)\in \mathbb{Z}_{+}^n$, define $\boldsymbol \alpha \circ \boldsymbol m:=(\alpha_1 m_1,\ldots, \alpha_n m_n).$ 
For any
$\boldsymbol{\alpha}=(\alpha_1,\ldots,\alpha_n)\in \{0,1\}^n,$ we shall denote $\{x\in\mathcal H:x\in\ker T_i \mbox{ if }\alpha_i=1\}$ 
by $\mathcal E_{\boldsymbol \alpha}$ with the convention that $\mathcal E_{0}=\mathcal H.$ 
The following result is a multivariate generalization of Theorem \ref{wold-type decomposition for pair}. 
\begin{theorem}\label{wold-type decomposition}
    Let $\boldsymbol{T}=(T_1,\ldots,T_n)$ be a left-inverse commuting $n$-tuple of bounded linear  operators defined on $\mathcal{H}$ such that for each $1\leq i \leq n,$ $T_i$ satisfies the Wold-type decomposition. Then  the $n$-tuple $\boldsymbol{T}$ has the joint Wold-type decomposition such that $\mathcal{H}=\oplus_{\boldsymbol{\alpha}\in \{0,1\}^n}~ \mathcal{H}_{\boldsymbol{\alpha}},$  
    where for each $\boldsymbol{\alpha} \in \{0,1\}^n,$
    \begin{eqnarray}\label{H-alpha}
        \mathcal{H}_{\boldsymbol{\alpha}} = \bigvee_{\boldsymbol{m}\in \mathbb{Z}_{+}^n} \boldsymbol{T}^{\boldsymbol{\alpha}\circ \boldsymbol{m}}\big(\bigcap_{\boldsymbol{m}\in \mathbb{Z}_{+}^{n}} \boldsymbol{T}^{\boldsymbol{\hat{\alpha}}\circ\boldsymbol{m}} \mathcal{E}_{\boldsymbol{\alpha}}\big)
    \end{eqnarray}
    is a joint $\boldsymbol{T}$-reducing subspace with $T_i$ on $\mathcal{H}_{\boldsymbol{\alpha}}$ is unitary if $\alpha_i=0$. 
\end{theorem}
\begin{proof}
    We prove this using the method of mathematical induction. The case when $n=2$ has been dealt with in Theorem \ref{wold-type decomposition for pair}. Suppose the result is true for any $k < n$. Let $(T_1,\ldots, T_k, T_{k+1})$ be any $(k+1)$-tuple of commuting operators satisfying the hypotheses of the proposed theorem. By the  induction assumption, 
     we have $2^k$ joint $(T_1,\ldots,T_k)$-reducing subspaces such that 
\begin{align*}    \mathcal{H}=\oplus_{\boldsymbol{\alpha}\in \mathbb{Z}_2^k} ~\mathcal{H}_{\boldsymbol{\alpha}},
\end{align*}
where $\mathcal H_{\boldsymbol\alpha}$ is as given in \eqref{H-alpha}.   
    It can be easily seen that all these subspaces are $T_{k+1}$-invariant. Hence, proceeding as in the proof of Theorem \ref{wold-type decomposition for pair}, we obtain that 
    \begin{eqnarray*}
        \mathcal H_{\boldsymbol\alpha} &=& \bigvee_{m\in \mathbb Z_+^{k+1}} \boldsymbol T^{(\boldsymbol \alpha, 1)\circ \boldsymbol m}\Big(\bigcap_{m\in \mathbb Z_+^{k+1}}\boldsymbol T^{(\hat{\boldsymbol \alpha},0)\circ \boldsymbol m}(\mathcal E_{(\boldsymbol \alpha, 1)})\Big) \oplus \bigvee_{m\in \mathbb Z_+^{k+1}} \boldsymbol T^{(\boldsymbol \alpha, 0)\circ \boldsymbol m}\Big(\bigcap_{m\in \mathbb Z_+^{k+1}}\boldsymbol T^{(\hat{\boldsymbol \alpha},1)\circ \boldsymbol m}(\mathcal E_{\boldsymbol \alpha})\Big)\\
        &=:& \mathcal H_{(\boldsymbol\alpha,1)}\oplus \mathcal H_{(\boldsymbol\alpha,0)}.
    \end{eqnarray*}
Therefore from the $2^k$ subspaces of the induction hypothesis, we get $2^{k+1}$ subspaces in total which are joint $(T_1,\ldots,T_{k+1})$-reducing. 
Note here that, as proved in Theorem \ref{wold-type decomposition for pair}, we get that $T_{k+1}$ is unitary on $\mathcal{H}_{(\alpha,0)}$. 
 This completes the proof. 
\end{proof}
Theorem \ref{wold-type decomposition}, in particular, proves \cite[Theorem 3.1]{JS} and \cite[Theorem 3]{Slocinski}.
\begin{rem}
    If $\boldsymbol T=(T_1,\ldots, T_n)$ is a $n$-tuple of left-inverse commuting operators, then $\boldsymbol {T'}:=(T_1',\ldots, T_n')$ is also a left-inverse commuting tuple.
\end{rem}
This remark, along with Theorem \ref{thm_2} proves the following multi-variable generalization of the Corollary 2.9 in \cite{Shimorin}. 
\begin{theorem}
    Let $(T_1, \ldots, T_n)$ be an $n$-tuple of left-inverse commuting operators on $\mathcal{H}$. Furthermore, assume that each $T_i$ has the Wold-type decomposition for $i=1,\ldots, n$. Then the $n$-tuple $(T_1', \cdots, T_n')$ has the Wold-type decomposition.  
\end{theorem}

Following \cite{SSS}, a commuting pair of operators $(T_1,T_2)$ is called a toral $2$-isometry if
 \begin{align}\label{Toral 2-isometry}
     I-T_i^*T_i-T_j^*T_j+T_j^*T_i^*T_iT_j = 0, \qquad i,j=1,2.
 \end{align}   
    The following proposition plays a central role in proving Theorem \ref{Model Theorem - Rank 1}, where we show that the operators $(M_{z_1}, M_{z_2})$ on $\mathcal{D}(\mu_1, \mu_2)$ as defined in \cite{SSS} become a model for the class of left-inverse commuting cyclic analytic toral $2$-isometric pair of bounded linear operators. It is worth noting that while Proposition \ref{inner product relation} is generally required as an assumption to prove \cite [Lemma 6.1(ii)] {SSS}, the result follows easily for the class of left-inverse commuting operators, and is stated below. Therefore, it turns out that Theorem \ref{Model Theorem - Rank 1} is certainly  a significant improvement over \cite[Theorem 2.4]{SSS}.
\begin{prop}\label{inner product relation}
    Let $(T_1,T_2)$
 be a left-inverse commuting pair of operators defined on a Hilbert space $\mathcal{H}$. Then for $x_0 \in ker T_1^* \cap kerT_2^*,$ we have the following:
 \begin{align*}
     &\langle T_1^m x_0, T_1^p T_2^q x_0 \rangle =0, ~~~q \geq 1, m,p \geq 0, \\
     &\langle T_2^n x_0, T_1^p T_2^q x_0 \rangle =0, ~~~ p \geq 1, n,q \geq 0.
 \end{align*}
 \end{prop}
 \begin{proof}
    The proof follows from Proposition \ref{proposition_1}. 
 \end{proof}
   
\begin{theorem}\label{Model Theorem - Rank 1}
    Let $\boldsymbol T=(T_1, T_2)$ be a commuting pair of operators which are left-inverse commuting on $\mathcal{H}.$ Assume $\boldsymbol{T}$ to be cyclic analytic toral $2$-isometry. Then there exist positive finite Borel measures $\mu_1,\mu_2$ on $\mathbb{T}$ such that $\boldsymbol T$ is unitarily equivalent with $(M_{z_1}, M_{z_2})$ on $\mathcal{D}(\mu_1, \mu_2)$.
\end{theorem}
\begin{proof}
    Since $T_1$, $T_2$ are analytic $2$-isometries, by \cite[Theorem 1]{Richter2}, both individually has the wandering subspace property.  Hence, they have the joint wandering subspace property due to \cite[Theorem 4.3]{MRV}. Moreover, \cite[Theorem 4.3]{MRV} asserts that $ker T_1^* \cap kerT_2^*$ is the corresponding generating wandering subspace for the pair $\boldsymbol T$ on $\mathcal{H}$, where \cite[Lemma 4.1]{MRV} proves that $ker T_1^* \cap kerT_2^*$ is a non-trivial closed subspace of $\mathcal{H}$. Choose a $x_0 \in ker T_1^* \cap kerT_2^*$.    
    For each $i=1,2$,  Consider $\mathcal{H}_i := \bigvee_{m \geq 0} T_i^m x_0$,  which is a $T_i$-invariant subspace. Since $T_i$ is analytic and $2$-isometry, $T_i \vert_{\mathcal{H}_i}$ is also analytic and $2$-isometry. Therefore, $T_i \vert_{\mathcal{H}_i}$ is cyclic analytic $2$-isometry and so from \cite [Theorem 5.1]{Richter1}, there exist positive finite Borel measures $\mu_i$ on $\mathbb{T}$ and unitary maps $U_i: \mathcal{H}_i \rightarrow \mathcal{D}(\mu_i)$ given by
    \begin{align*}
      U_iT_i = M_z U_i ~~~~~  \text{and}~~~~ U_i x_0=1.
    \end{align*}
    Here $M_z$ represents the multiplication by the coordinate function $z$ on $\mathcal{D}(\mu_i)$. Define a linear map $U: \mathcal{H} \rightarrow \mathcal{D}(\mu_1,\mu_2)$ given as in \cite[Theorem 2.4]{SSS} by the rule 
    \[U(T_1^mT_2^nx_0)=z_1^m z_2^n, ~~m,n \geq 0.\]
    We claim that the map $U$ is an unitary from $\mathcal{H}$ onto $\mathcal{D}(\mu_1,\mu_2).$
    Infact, for $x \in \mathcal{H}$ given by $x=\sum_{m=0}^k \sum_{n=0}^l a_{m,n} T_1^m T_2^nx_0$, we have 
\begin{align*}
    UT_1x &=UT_1\Big(\sum_{m=0}^k \sum_{n=0}^l a_{m,n} T_1^m T_2^n x_0\Big) =U\Big(\sum_{m=0}^k \sum_{n=0}^l a_{m,n} T_1^{m+1} T_2^n x_0 \Big)\\
    &= \sum_{m=0}^k \sum_{n=0}^l a_{m,n} z_1^{m+1} z_2^n = z_1 \sum_{m=0}^k \sum_{n=0}^l a_{m,n} z_1^{m} z_2^n\\
    & 
    = M_{z_1} Ux. 
\end{align*}
    Similarly, we also have $UT_2= M_{z_2} U$.
   
    Now, we proceed to show that $U$ is indeed a unitary operator. To this end note that, since $T$ is cyclic with cyclic vector $x_0$, we have $\mathcal{H}=\bigvee_{m,n \geq 0}T_1^m T_2^n x_0$. Due to the fact that the set of polynomials is dense in $\mathcal{D}(\mu_1,\mu_2)$ (see \cite[Theorem 2.1]{SSS}), we also have $\mathcal{D}(\mu_1,\mu_2)= \bigvee_{m,n \geq 0} z_1^mz_2^n$.
    Hence, it is enough to show that, for any $m,n \in \mathbb{Z_{+}}$
\begin{align}\label{unitary relation}
    \langle T_1^m T_2^n x_0, T_1^p T_2^qx_0 \rangle = \langle z_1^m z_2^n, z_1^p z_2^q \rangle.  
\end{align}
    Consider
\begin{align*}
    \langle T_i^m x_0, T_i^n x_0 \rangle  &= \langle U_iT_i^m x_0, U_i T_i^n x_0 \rangle _{\mathcal{D}(\mu_i)} \\
    &= \langle M_z^mU_ix_0, M_z^nU_ix_0 \rangle_{\mathcal{D}(\mu_i)} \\ & = \langle z^m, z^n \rangle_{\mathcal{D}(\mu_i)} \\
    &= \langle z_i^m, z_i^n \rangle_{\mathcal{D}(\mu_1, \mu_2)}.
\end{align*}
    Since $T$ is a left-inverse commuting  toral $2$-isometry, using Proposition \ref{inner product relation}, we observe that the hypothesis of  Lemma \cite[6.1(ii)]{SSS} is satisfied. Therefore, the conclusion of \cite[Lemma 6.1(ii)]{SSS} immediately follows. This proves \eqref{unitary relation} and completes the proof.
\end{proof}

\section{Model Space $\mathcal{D}_{\mathcal{E}}(\mu_1,\mu_2)$ and Left-inverse Commuting Toral $2$-isometries }\label{Dirichlet space - vector valued}

    Let $\mathcal{E}$ be any Hilbert space and let $\mu_1,\mu_2$ be any two positive $\mathcal{L}(\mathcal{E})$-valued operator measures defined on $\mathbb{T}$. Denote the Dirichlet-type space corresponding to $\mu_1,\mu_2 \in \mathcal{L}(\mathcal{E})$  over the bidisc $\mathbb{D}^2$ by $\mathcal{D}_{\mathcal{E}}(\mu_1,\mu_2)$ and is defined by
    \[\mathcal{D}_{\mathcal{E}}(\mu_1,\mu_2):=\{f \in \mathcal{O}(\mathbb{D}^2,\mathcal{E}): D_{\mathcal{E}}(\mu_1,\mu_2)(f) < \infty \}, \]
    where $\mathcal O(\mathbb D^2, \mathcal E)$ is the set of all $\mathcal E$-valued holomorphic maps on the bidisc $\mathbb D^2$ and 
\begin{align*}
    D_{\mathcal{E}}(\mu_1,\mu_2)(f)&:=\lim_{r \rightarrow 1^-} \int_{\mathbb{D}}\int_{\mathbb{T}}\langle P_{\mu_1}(z_1)\partial_1f(z_1,re^{it}), \partial_1f(z_1, re^{it})\rangle dA(z_1) dt \\ &+  \lim_{r \rightarrow 1^-}\int_{\mathbb{D}}\int_{\mathbb{T}}\langle P_{\mu_2}(z_2)\partial_2f(re^{it},z_2), \partial_2f(re^{it},z_2)\rangle dA(z_2) dt.  
\end{align*}    
    For any $f \in \mathcal{D}_{\mathcal{E}}(\mu_1,\mu_2)$, the norm of $f$ is given by $\|f\|^2_{\mathcal{D}_{\mathcal{E}}(\mu_1, \mu_2)}:=\|f\|^2_{H^2(\mathbb{D}^2)}+ D_{\mathcal{E}}(\mu_1,\mu_2)(f).$ 
    The space $\big(\mathcal{D}_{\mathcal{E}}(\mu_1,\mu_2), \|\cdot\|_{\mathcal{D}_{\mathcal{E}}(\mu_1, \mu_2)}\big)$ becomes a Hilbert space. Corresponding inner product shall be denoted by $\langle \cdot , \cdot \rangle_{\mathcal D_{\mathcal{E}}(\mu_1,\mu_2)}.$
    The following are some of the  important properties regarding  $\mathcal{D}_{\mathcal{E}}(\mu_1,\mu_2)$. We refer the reader to \cite{SSS} for an analogous proof of these facts.
\begin{theorem}\cite [Lemma 1.3, Theorem 2.1]{SSS}
    Let $\mu_1,\mu_2$ be any two positive $\mathcal{L}(\mathcal{E})$-valued operator measures defined on a Hilbert space $\mathcal{E}$. Then the following holds:
\begin{enumerate}
    \item The space $\big(\mathcal{D}_{\mathcal{E}}(\mu_1,\mu_2), \|\cdot\|_{\mathcal{D}_{\mathcal{E}}(\mu_1, \mu_2)}\big)$ is a reproducing kernel Hilbert space.
    \item The operators of multiplication by the coordinate functions $z_1,z_2$, i.e., $M_{z_1},M_{z_2}$ are bounded linear operators on $\mathcal{D}_{\mathcal{E}}(\mu_1,\mu_2)$.
    \item The pair of operators $(M_{z_1}, M_{z_2})$ on $\mathcal{D}_{\mathcal{E}}(\mu_1, \mu_2)$ is  toral $2$-isometry, i.e., for every $f \in \mathcal{D}_{\mathcal{E}}(\mu_1, \mu_2)$ we have 
   \begin{align*}
       \|z_i z_jf\|^2_{\mathcal{D}_{\mathcal{E}}(\mu_1, \mu_2)} - \|z_if\|^2_{\mathcal{D}_{\mathcal{E}}(\mu_1, \mu_2)} -\|z_jf\|^2_{\mathcal{D}_{\mathcal{E}}(\mu_1, \mu_2)} + \|f\|^2_{\mathcal{D}_{\mathcal{E}}(\mu_1, \mu_2)} = 0, \quad i,j=1,2.
   \end{align*}
      \item The set of polynomials forms a dense subset in $\mathcal{D}_{\mathcal{E}}(\mu_1,\mu_2).$
\end{enumerate}
\end{theorem}

    We note here the following formula regarding the inner product of monomials in $\mathcal{D}_{\mathcal{E}}(\mu_1,\mu_2)$. The following lemma can be easily verified from the definition of norm in $\mathcal{D}_{\mathcal{E}}(\mu_1,\mu_2)$ and can be proven analogously as in \cite[Proposition 3.10]{SSS}. We omit the proof.
\begin{lemma}\label{inner product - monomials}
    Suppose $\mu_1,\mu_2$ are two positive $\mathcal{L}(\mathcal{E})$-valued measures. Then, for any $x,y\in \mathcal E$ and $m,n,p,q\in\mathbb Z_+,$ we have
\[\langle xz_1^m z_2^n, yz_1^p z_2^q \rangle_{\mathcal D_{\mathcal E}(\mu_1,\mu_2)} =\begin{cases}
0, & \text{if $m \neq p$, $n \neq q$},\\
      (n \wedge q)\langle \hat{\mu_2}(q-n)x, y \rangle & \text{if $m = p, n \neq q$}, \\
       (m \wedge p)\langle \hat{\mu_1}(p-m)x, y\rangle & \text{if $m \neq p, n =q$}, \\
       1+m\langle \hat{\mu_1}(0)x,y \rangle + n \langle \hat{\mu_2}(0)x, y \rangle  & \text{if $m=p,n =q$.}
\end{cases}
\]
\end{lemma}
    The following result generalizes Theorem \ref{Model Theorem - Rank 1} and establishes that the pair $(M_{z_1},M_{z_2})$ on $\mathcal{D}_{\mathcal{E}}(\mu_1,\mu_2)$ serves as a canonical model for the class of pairs of left-inverse commuting analytic toral $2$-isometries.

\begin{theorem}\label{Model Theorem - Higher Rank}
    Consider $\boldsymbol T=(T_1,T_2)$ to be a left-inverse commuting pair of analytic toral $2$-isometries on $\mathcal{H}$. Then there exist positive  $\mathcal{L}(\mathcal{E})$-valued operator measures $\mu_1,\mu_2$ on $\mathbb{T}$ such that $\boldsymbol T$ is unitarily equivalent with $(M_{z_1}, M_{z_2})$  on $\mathcal{D}_{\mathcal{E}}(\mu_1,\mu_2)$.
\end{theorem}

\begin{proof}
    Let $j\in\{1,2\}.$ Define $\mathcal{H}_j:=\bigvee_{n \geq 0}T_j^n(\mathcal E),$  which is a $T_j$-invariant subspace, where we choose to denote $\mathcal{E}:=\bigcap_{i=1}^2ker T_i^*$. Since $\boldsymbol{T}$ is left-inverse commuting, by \cite[Lemma 4.1] {MRV}, we obtain that the subspace $\mathcal{E}$ is a non-trivial closed subspace of $\mathcal{H}$. As $\boldsymbol{T}$ is   analytic toral $2$-isometry, it follows that the operator $T_j \vert_{\mathcal{H}_j}$ is an analytic $2$-isometry. Therefore by \cite[Theorem 4.1]{Olofsson} there exists spectral measure $E_j$ on $\mathbb T$ such that
     $T_j\vert_{\mathcal{H}_j}$ is unitarily equivalent to the multiplication by the co-ordinate function $M_z$ on $\mathcal{D}_{\mathcal{E}}(\mu_j)$ (for definition see \cite{Olofsson}), where for any Borel subset $\sigma$ of $\mathbb T,$  $\mu_j(\sigma) = PD_jP_{\mathcal{D}_j}E_j(\sigma)D_j\vert_{\mathcal{E}},$ 
     $P$ is the orthogonal projection of $\mathcal H$ onto $\mathcal{E},$ $D_j,$ $\mathcal D_j$ and $P_{\mathcal D_j}$ denote the defect operator, defect space, and orthogonal projection of $\mathcal H$ onto defect space corresponding to the operator $T_j$ respectively.
    That is, there exists a map $U_j: \mathcal{H}_j \rightarrow \mathcal{D}_{\mathcal{E}}(\mu_j)$ which is unitary and that  $U_jT_j=M_z U_j$. 
    In fact, as it is clear from \cite[Equation 4.3]{Olofsson} that ${U_j}_{\vert \mathcal E}=I_\mathcal E$.  This fact shall be used in what follows.
    Consider the Dirichlet-type space  $\mathcal{D}_{\mathcal{E}}({\mu_1}, {\mu_2})$, where ${\mu_j}$ for $j=1,2$, are defined as above, (see also \cite{MRV_2}).
 For each $j=1,2$, $T_j$ is analytic $2$-isometry. Hence by \cite [Theorem 1]{Richter2}, $T_j$ has the wandering subspace property. That is, we have
\begin{align*}
    \mathcal{H}=\bigvee_{n \geq 0} T_j^n(ker T_j^*), ~~\text{for each} ~~j=1,2.
\end{align*}
     Thus by \cite[Theorem 4.3]{MRV}, we get
\begin{align}\label{joint wandering subspace property}
    \mathcal{H}=\bigvee_{m,n \geq 0} T_1^m T_2^n \mathcal{E}.
\end{align}
    For any $x \in \mathcal{H}$ of the form $x = \sum_{k=0}^m \sum_{l=0}^n T_1^k T_2^la_{k,l}$ where $a_{k,l} \in \mathcal{E}$, we
    define the map $U: \mathcal{H} \rightarrow \mathcal{D}_{\mathcal{E}}({\mu_1},{\mu_2})$ by  $x \mapsto \sum_{k=0}^m \sum_{l=0}^n a_{k,l}z_1^k z_2^l$. By \eqref{joint wandering subspace property} and that $\mathcal{D}_{\mathcal{E}}({\mu_1},{\mu_2}) = \bigvee_{m,n \geq 0} a_{m,n}z_1^m z_2^n$, it suffices to show that for any $m,n,p,q \geq 0$ and $a_{m,n}, a_{p,q}\in \mathcal E,$ 
\begin{align}\label{isometry - U}
    \langle T_1^m T_2^n a_{m,n}, T_1^pT_2^q a_{p,q} \rangle = \langle a_{m,n} z_1^m z_2^n, a_{p,q} z_1^p z_2^q \rangle.
\end{align}
    The case when $m \neq p, n \neq q,$ \eqref{isometry - U} follows from Lemma \ref{inner product - monomials} easily, since the equality $\langle T_1^m T_2^n a_{m,n}, T_1^p T_2^q a_{p,q} \rangle =0$ can be deduced from \cite[Lemma 6.1(ii)]{SSS} and Proposition \ref{inner product relation}.
     Consider the case when $m = p, n \neq q$. Again using \cite [Lemma 6.1(ii)]{SSS}, we have 
     $$\langle T_1^m T_2^n a_{m,n}, T_1^m T_2^q a_{m,q}\rangle = \langle T_2^n a_{m,n}, T_2^q a_{m,q} \rangle.$$ 
     By repetitive applications of \cite [Lemma 6.1(ii)]{SSS}, Proposition \ref{inner product relation}, and Lemma \ref{inner product - monomials}, we get 
\begin{align*}
    \langle T_2^n a_{m,n,}, T_2^q a_{m,q} \rangle  &= \langle U_2T_2^n a_{m,n}, U_2T_2^q a_{m,q} \rangle_{\mathcal{D}_{\mathcal{E}}(\mu_2)} \\
    &= \langle M_z^nU_2 a_{m,n}, M_z^qU_2 a_{m,q} \rangle_{\mathcal{D}_{\mathcal{E}}(\mu_2)} \\
    &= \langle z^n a_{m,n}, z^q a_{m,q} \rangle_{\mathcal{D}_{\mathcal{E}}(\mu_2)} \\
    &= \min (n,q) \langle \hat{\mu}_2(q-n)a_{m,n}, a_{m,q} \rangle \\
    &= \min(n,q) \langle PT_2^{*(q-n)}(T_2^*T_2-I)\vert_{\mathcal{E}} ~a_{m,n}, a_{m,q} \rangle \\
    &= \langle z_2^n a_{m,n}, z_2^q a_{m,q} \rangle_{\mathcal{D}_{\mathcal{E}}({\mu_1},{\mu_2})}.
\end{align*}
    Similarly equation \eqref{isometry - U}  can be verified in the case when $m \neq p, n =q$. 
    For $m=p, n=q$, by \cite[Lemma 6.1(ii)] {SSS}, we have
\begin{align*}
    \|T_1^m T_2^n a_{m,n}\|^2 &= \|T_1^ma_{m,n}\|^2+\|T_2^na_{m,n}\|^2 - \|a_{m,n}\|^2 \\
    & = \|U_1T_1^m a_{m,n}\|^2 + \|U_2 T_2^n a_{m,n}\|^2 - \|U_1a_{m,n}\|^2 \\
    &= \|M_z^m U_1a_{m,n}\|^2 + \|M_z^n U_2a_{m,n} \|^2 - \|a_{m,n}\|^2 \\
    &= \|z^m a_{m,n}\|^2_{\mathcal{D}_{\mathcal{E}}({\mu_1})} + \|z^n a_{m,n}\|^2_{\mathcal{D}_{\mathcal{E}}({\mu_2})} - \|a_{m,n}\|^2_{\mathcal{D}_{\mathcal{E}}({\mu_1})} \\
    &= \|z_1^ma_{m,n}\|^2_{\mathcal{D}_{\mathcal{E}}({\mu_1}, {\mu_2})}+\|z_2^n a_{m,n}\|^2_{\mathcal{D}_{\mathcal{E}}({\mu_1},{\mu_2})}-\|a_{m,n}\|^2_{\mathcal{D}_{\mathcal{E}}({\mu_1},{\mu_2})} \\
    &= \|a_{m,n} z_1^m z_2^n \|^2_{{\mathcal{D}_{\mathcal{E}}({\mu_1},{\mu_2})}}. 
\end{align*}
This completes the proof.
\end{proof}
\begin{theorem}\label{Wold-type decomposition for pair- vector valued}
    Let $(T_1, T_2)$ be a pair of left-inverse commuting toral $2$-isometries. 
    Then there exist positive operator valued measures $\mu_1,\mu_2,\nu_1$ and $\nu_2$ on $\mathbb{T}$ such that
    \[\mathcal{H} \cong \mathcal{H}_{00} \oplus \mathcal{D}_{\mathcal{E}_{01}}(\mu_1) \oplus \mathcal{D}_{\mathcal{E}_{10}}(\mu_2) \oplus \mathcal{D}_{\mathcal{E}}(\nu_1,\nu_2) \]
    where $\mathcal{H}_{00}=\cap_{m,n \geq 0}T_1^m T_2^n \mathcal{H}$, $\mathcal{E}_{01}=\cap_{m \geq 0}T_1^m kerT_2^*$, $\mathcal{E}_{10}=\cap_{n \geq 0}T_2^n kerT_1^*$,   $\mathcal{E}=ker T_1^* \cap kerT_2^*$. Moreover $T_1$ is unitary on $\mathcal{H}_{00}$ and $\mathcal{D}_{\mathcal{E}_{01}}(\mu_1),$ and $T_2$ is unitary on $\mathcal{H}_{00}$  and $\mathcal{D}_{\mathcal{E}_{10}}.$ That is, we have the following decomposition 
\[ T_1 \cong
\begin{bmatrix}
    U_0 & 0 & 0 & 0 \\
    0 & U_1 & 0 & 0 \\
    0 & 0 & M_z & 0 \\
    0 & 0 & 0 & M_{z_1}
\end{bmatrix}
\text{  and  }  T_2 \cong
\begin{bmatrix}
    V_0 & 0 & 0 & 0 \\
    0 & M_z & 0 & 0 \\
    0 & 0 & V_1 & 0 \\
    0 & 0 & 0 & M_{z_2}
\end{bmatrix}
\]
where $U_0,V_0$ are unitary on $\mathcal{H}_{00}$ and $U_1,V_1$ are respective unitaries on $\mathcal{D}_{\mathcal{E}_{01}}(\mu_1), \mathcal{D}_{\mathcal{E}_{10}}(\mu_2)$.
\end{theorem}    
   \begin{proof}
   As for each $i=1,2,$ the operator $T_i$ is a $2$-isometry, therefore it admits a Wold-type decomposition by  \cite [Theorem 3.6] {Shimorin}. Now, an application of Theorem \ref{wold-type decomposition for pair} completes the proof of the theorem.
\end{proof}

\noindent\textbf{Acknowledgment:}
The first named author's research is supported by the DST-INSPIRE Faculty Grant with Fellowship No. DST/INSPIRE/04/2020/001250, and the third named author's research work is supported by the BITS Pilani institute fellowship.

\end{document}